\documentclass[11pt]{article}
\usepackage[colorlinks=true,
urlcolor=blue,linkcolor=blue,citecolor=blue]{hyperref}
\usepackage{amsmath,amsfonts,amssymb,amsthm}
\usepackage{mathrsfs,mathtools}
\usepackage{color}
\usepackage{amscd}

\newtheorem{theorem}{Theorem}[section]
\newtheorem{lemma}[theorem]{Lemma}
\newtheorem{corollary}[theorem]{Corollary}

\newtheorem{definition}[theorem]{Definition}

\theoremstyle{remark}

\numberwithin{equation}{section}

\def\V{{\mathfrak V}}
\def\X{{\mathfrak X}}

\def\E{{\mathbb{E}}}            

\def\eps{\varepsilon}

\def \Rm {\mathbb R}

\def \Cm {\mathbb C}

\renewcommand{\P}{\mathbb P}

\newcommand{\dint}{\displaystyle\int}

\newcommand{\cout}[1]{}

\newcommand{\mf}{\mathfrak f}

\newcommand{\p}{\mathfrak p}

\newcommand{\mF}{{\mathcal F}}
\newcommand{\mH}{{\mathcal H}}

\newcommand{\LL}{{\mathbb L}}

\newcommand{\rD}{{\rm D}}

\title{Homogenization of elliptic equations with large random potential}

\author{Guillaume Bal
\thanks {Department of Applied Physics and Applied
Mathematics, Columbia University, New York NY, 10027, (gb2030@columbia.edu).}
\and Ningyao Zhang \thanks {Department of Applied Physics and Applied
Mathematics, Columbia University, New York NY, 10027, (nz2164@columbia.edu).}}
\begin{document}
\maketitle
\begin{abstract}
We consider an elliptic equation with purely imaginary, highly heterogeneous, and large random potential with a sufficiently rapidly decaying correlation function. We show that its solution is
well approximated by the solution to a homogeneous equation with
a real-valued homogenized potential as the correlation length of the random medium
$\varepsilon\rightarrow 0$ and estimate the size of the random fluctuations in the setting $d\geq3$.
\end{abstract}

\section{Introduction}
We study the asymptotic behavior of the solution to the equations parameterized by $\varepsilon$
\begin{equation}\label{eq:ell}
    \big(-\Delta+1-iV_\eps\big)u_{\varepsilon}(x)=f(x)
\end{equation}
for $x\in \mathbb{R}^d$ as $\varepsilon\rightarrow 0$ in dimension $d\geq 3$ with $V_{\varepsilon}={\varepsilon}^{-1}V(\frac{x}{\varepsilon})$. Here, $i=\sqrt{-1}$. As a possible application for \eqref{eq:ell}, we may rewrite it as the system
\begin{equation}
\begin{pmatrix}
-\Delta+1 & 0 \\0 & -\Delta+1 \\
\end{pmatrix}
\begin{pmatrix}
u_{1,\varepsilon}\\u_{2,\varepsilon}
\end{pmatrix}
+\begin{pmatrix}
0 & V_{\varepsilon}\\-V_{\varepsilon} & 0
\end{pmatrix}
\begin{pmatrix}
u_{1,\varepsilon}\\u_{2,\varepsilon}
\end{pmatrix}
=
\begin{pmatrix}
f_r\\f_i
\end{pmatrix},
\end{equation}
where we have defined $f=f_r+if_i$ and where $V_{\varepsilon}$ may model the (linear) interaction between two populations represented by the densities $u_{1,\eps}$ and $u_{2,\eps}$. In the absence of interactions, the two populations follow independent diffusions. Assuming that the interaction is modeled by a large, highly oscillatory, random, and mean zero field  $V_\eps$, we wish to understand the limit as the correlation length $\eps\to0$ of such interactions.

It turns out that the limiting behavior of $u_\eps$ depends on the correlation properties of $V$. When the latter decay slowly (of the form $|x|^{-\gamma}$ as $|x|\to\infty$ with $\gamma<2$), we expect $u_\eps$ to converge to the solution of a stochastic partial differential equation; see \cite{B-CMP-09,ZB-CMS-13} for such results in a time-dependent setting.  In dimension $d=1$, we also expect the solution $u_\eps$ to remain stochastic in the limit $\eps\to0$ \cite{PP-GAK-06}. We consider here the setting where the correlation function decays sufficiently rapidly so that $u_\eps$ is expected to converge to a deterministic, homogenized, solution. The main objective of this paper is to present such a convergence result in the setting $d\geq3$ and to provide an optimal rate of convergence when the potential $V$ is assumed to be sufficiently mixing. A similar result, not considered here, is expected to hold in the critical dimension $d=2$ with the strength of the random potential $\eps^{-1}$ in \eqref{eq:ell} replaced by $\eps^{-1}|\ln\eps|^{-\frac12}$ \cite{B-MMS-10}.

The above problems are written on $\Rm^d$ to simplify the presentation. Our convergence result would also hold for a  problem posed on a bounded open domain $X$ with, say, Dirichlet conditions on $\partial X$. The operator $-\Delta+1$ could also be replaced by any operator of the form $-\nabla\cdot a\nabla + b$ with $a$ (as a symmetric tensor) and $b$ sufficiently smooth and bounded above and below by positive constants.

The homogenization of partial differential equations in periodic or random media has a long history; see for instance \cite{A-SP-02,BLP-77,JKO-SV-94}. The homogenization of elliptic equations with random diffusion coefficients was treated in \cite{Koz-MSb-79,PV-RF-81}. Rates of convergence to homogenization in similar settings are proposed in \cite{CS-IM-10,CN-EJP-00,GO-AP-11,Yur-SM-86}.  The homogenization of elliptic and parabolic equations with large random potential has also been studied recently in different contexts. Convergence to stochastic limits is considered in \cite{B-CMP-09,PP-GAK-06,ZB-CMS-13}. Convergence to homogenized solutions is treated in \cite{B-MMS-10,ZB-SD-13} by diagrammatic techniques, in \cite{GB-INV-13,GB-WEAK-13} using probabilistic representations,  and in \cite{HPP-SPDE-13} using a multi-scale method ; see also the review \cite{BG-CMS-14}.


We now present our main hypotheses on the potential $V$ and our main results.

The potential $V(x,\omega)$ is defined, following \cite{PV-RF-81}, on an abstract probability space $(\Omega,\mF,\P)$ with $\V(\omega)$ a bounded measurable function on $\Omega$. We assume the existence of a translation group $\tau_x:\Omega\to\Omega$ for all $x\in\Rm^d$ leaving $\P$ invariant and being ergodic in the sense that for all $A\in\mF$ such that $\tau_xA\subset A$ for all $x\in\Rm^d$, then either $\P(A)=0$ or $\P(A)=1$. Let $\mH=L^2(\Omega,\mF,\P)$. For $\mf\in \mH$ and $x\in\Rm^d$, we define the unitary operator $T_x$ on $\mH$ as $T_x\mf(\omega)=\mf(\tau_{-x}\omega)$. The stationary, bounded, potential $V$ is then defined as $V(x,\omega)=T_x \V(\omega)=\V(\tau_{-x}\omega)$. The group (in $x$) of unitary operators $T_x$ admits a spectral resolution
\begin{displaymath}
   T_x = \dint_{\Rm^d} e^{i\xi\cdot x} U(d\xi)
\end{displaymath}
for $U(d\lambda)$ the associated projection valued measure and the $s$ powers of the (positive) Laplacian $\LL$ are given by
\begin{displaymath}
  \LL^{\frac s2} = \dint_{\Rm^d} |\xi|^s U(d\xi).
\end{displaymath}
Note that for $f(x,\omega)=T_x\mf(\omega)$, we have $(-\Delta)^{\frac s2}f(x,\omega)=T_x \LL^{\frac s2}\mf(\omega)=\LL^{\frac s2}f(x,\omega)$, $dx\times\P-$a.s., where $\Delta$ is the usual (negative) Laplacian in $\Rm^d$.

The correlation function of $\V$ (and $V$) is defined as
\begin{equation}
\label{eq:corr}R(x)=\E\{ \V T_x\V\} = \E\{V(0,\cdot)V(x,\cdot)\}.
\end{equation}
The power spectrum $\hat R(\xi)$ is the (rescaled) Fourier transform of $R$ defined by
\begin{equation}
\label{eq:powerspectrum}
 (2\pi)^d \hat R(\xi) = \dint_{\Rm^d} e^{-i\xi\cdot x} R(x) dx.
\end{equation}
The main assumption we make on the correlation function is that
\begin{equation}
\label{eq:rho}
\rho:= \dint_{\Rm^d} \dfrac{\hat R(\xi)}{|\xi|^2} d\xi<\infty.
\end{equation}
This may be recast by Plancherel as $\rho=\int_{\Rm^d}\Phi(x)R(x)dx$ for $\Phi(x)=\Delta^{-1}\delta$ the fundamental solution to the Laplace equation in dimension $d\geq3$. Such an assumption is satisfied when $R(x)$ decays like $\kappa|x|^{-\gamma}$ as $|x|\to\infty$, or equivalently when $\hat R(\xi)$ behaves as $\kappa'|\xi|^{\gamma-d}$ as $\xi\to0$, with $\gamma>2$.

The bound \eqref{eq:rho} is the main hypothesis we impose on $\V$, beyond stationarity and ergodocity. When the latter fails, for instance when $\gamma<2$ in the above example, then we do not expect $u_\eps$ to converge to a homogenized solution \cite{B-CMP-09,PP-GAK-06,ZB-CMS-13}. For technical reasons, we also need in the convergence result to make some regularity assumptions on $\V$ and assume that $\V_s:=\LL^{\frac s2}\V$ satisfies the same hypothesis as $\V$ for some $s>\frac{d-2}4$. By construction, the power spectrum of $\V_s$ is given by $|\xi|^{2s}\hat R(\xi)$ so we also impose that $|\xi|^{2s-2}\hat R(\xi)$ is integrable. With these hypotheses, we can state the following result
\begin{theorem}
\label{thm:conv}
Let us assume that $V$ is a stationary, bounded, random field such that
\begin{equation}
\label{eq:ctes}
\dint_{\Rm^d}\dfrac{1+|\xi|^{2s}}{|\xi|^2} \hat R(\xi)<\infty \quad \mbox{ for some } \quad s>\frac{d-2}4.
\end{equation}
Then $u_\eps$ the unique solution to \eqref{eq:ell} with $f\in H^{-1}(\Rm^d)$ converges weakly in $H^1(\Rm^d;\mH)$ and strongly in $L^2_{\rm loc}(\Rm^d;\mH)$ to the unique solution of the deterministic equation
\begin{equation}
\label{eq:limit}
   -\Delta u + (1+\rho) u = f,\quad \Rm^d
\end{equation}
 with $\rho$ defined in \eqref{eq:rho}.
\end{theorem}

When the decay rate of the correlation function $R$ is sufficiently large and $\V$ satisfies additional technical assumptions, then we obtain an optimal rate of convergence of $u_\eps$ to $u$ in $H^1(\Rm^d;\mH)$. More precisely, we now assume that $V$ is bounded $\P-$a.s. (although this specific bound does not appear in subsequent estimates), that $R(x)\in L^1(\Rm^d)$ and that for all $(x_1,x_2,x_3,x_4)\in (\Rm^d)^4$,
\begin{equation}
\label{eq:estim4th}
\big| \E\{ \prod_{j=1}^4 V(x_j)\} - R(x_1-x_2)R(x_3-x_4) \big|
\leq \eta(|x_1-x_3|)\eta(|x_2-x_4|) + \eta(|x_1-x_4|)\eta(|x_2-x_3|),
\end{equation}
for some integrable function $\eta$ from $\Rm_+$ to $\Rm_+$.

A large class of mixing potentials with sufficiently rapidly decaying maximal correlation function was shown to satisfy \eqref{eq:estim4th} in \cite{HPP-SPDE-13}; see also \cite{BJ-CMS-11} for similar bounds for specific distributions.
%
Our main convergence result is then the following theorem.
\begin{theorem}\label{thm:rate}
We assume that $V$ is bounded, that the correlation function $R(x)\in L^1(\Rm^d)$ and that there is an integrable function $\eta$ such that \eqref{eq:estim4th} holds.
In dimension $d\geq 3$, the solution $u_{\varepsilon}(x)$ to \eqref{eq:ell} with $f\in L^2(\Rm^d)$ converges to the solution $u$ of \eqref{eq:limit}. Moreover, we have the estimate
\begin{equation}
\|u_{\varepsilon}-u\|_{L^2(\mathbb{R}^d;\mH)} \leq C \left\{
\begin{array}{cl}\sqrt{\varepsilon} & d=3\\
\varepsilon\sqrt{|\ln\varepsilon|} & d=4\\
\varepsilon & d>4.
\end{array}
\right.
\end{equation}
\end{theorem}
In fact, for a vector field $\Xi(x)$ in $L^2(\Rm^d)$ formally defined as $-\nabla \Delta^{-1} V$ (see Lemma \ref{lem:psieps} for a more precise statement), then we obtain that $\|\nabla u_\eps - \nabla u -u \Xi(\frac \cdot\eps)\|_{L^2(\mathbb{R}^d;\mH)}$ satisfies the same bound as $\|u_{\varepsilon}-u\|_{L^2(\mathbb{R}^d;\mH)}$.

The rest of the paper is organized as follows. The proof of theorem \ref{thm:conv} is presented in section \ref{sec:conv}.
The proof of theorem \ref{thm:rate} is given in section \ref{sec:rate} with technical calculations involving fourth moments postponed to section \ref{sec:fourth}.


\section{Energy and perturbed test function methods}
\label{sec:conv}

Let us consider the problem \eqref{eq:ell} with $f\in H^{-1}(\Rm^d):=H^{-1}(\Rm^d;\Cm)$. We assume that $V$ is bounded on $\Rm^d$ $\P-$a.s. to simplify the presentation. Multiplying the equation by $u_\eps^*$ with $u_\eps\in H^1(\Rm^d):=H^1(\Rm^d;\Cm)$ solution of the above equation and integrating by parts gives us the a priori estimate
\begin{equation}
\label{eq:intparts}
 \dint_{\Rm^d} \Big(|\nabla u_\eps|^2 + |u_\eps|^2 -iV_\eps |u_\eps|^2\Big) dx = \dint_{\Rm^d} f u_\eps^* dx.
\end{equation}
Upon taking the real part, we obtain by Cauchy-Schwarz that
\begin{equation}
\label{eq:apriori}
 \|u_\eps \|_{H^1(\Rm^d)} \leq \|f\|_{H^{-1}(\Rm^d)} \quad \P-a.s.
\end{equation}
By the Lax-Milgram theory, we thus obtain that \eqref{eq:ell} admits a unique solution in $H^1(\Rm^d)$ $\P-$a.s. for any source $f\in H^{-1}(\Rm^d)$. Note that when the source $f\in H^{-1}(\Rm^d;\mH)$ (defined as the dual to $H^1(\Rm^d;\mH)$), then the solution $u_\eps$ is bounded in $H^1(\Rm^d;\mH)$ by the preceding estimate.

From the previous estimate, we deduce that $u_\eps$ converges weakly in $H^1(\Rm^d)$ $\P-$a.s. to a limit $u\in H^1(\Rm^d)$ (after possible extraction of a subsequence, though the limit $u$ will be proved to be unique and hence the whole sequence converges). Moreover, for $\theta$ a smooth function with compact support, we have by the Rellich-Kondrachov embedding that $\theta u_\eps$ converges strongly in $L^p(\Rm^d)$ to its limit $\theta u$ for all $1\leq p<\frac {2d}{d-2}$. Our aim is now to pass to the limit in a variant of \eqref{eq:intparts} and obtain the limiting equation for $u$.

Let $\theta_\eps\in H^1(\Rm^d;\mH)$ be a (complex-valued) test function. We thus find that
\begin{equation}
\label{eq:intpartheta}
 \E\dint_{\Rm^d} (\nabla u_\eps\cdot\nabla\theta_\eps^* + u_\eps \theta_\eps^* -iV_\eps u_\eps\theta_\eps^*) dx = \E\dint_{\Rm^d} f \theta_\eps^* dx.
\end{equation}
In order to pass to the limit in the above expression, we need to replace the highly oscillatory $V_\eps$ by a better-behaving function, and as it turns out, we need to choose $\theta_\eps$ as an $\eps-$dependent function to help cancel out large contributions.

Our first task is to replace $V_\eps$ by an object of the form $\Delta \psi_\eps$ so that after integrations by parts, the resulting $\nabla\psi_\eps$ is bounded in an appropriate manner as $\eps\to0$. We introduce the {\em corrector} (following standard terminology in homogenization theory) $\psi_\eps$ solution of
\begin{equation}
\label{eq:psieps}
  (-\Delta +1) \psi_\eps + V_\eps =0.
\end{equation}
By an application of the Lax-Milgram lemma, the real-valued function $\psi_\eps$ is uniquely defined in $H^1(\Rm^d;\mH)$. Moreover, in the variables $y=\frac x\eps$, let us define
\begin{equation}
\label{eq:psiepsy}
 (-\Delta+\eps^2) \psi^\eps (y) + V(y)=0.
\end{equation}
Therefore, $\psi^\eps$ is morally an approximation of $\Delta^{-1}V$, which is not defined and thus regularized with the small absorption coefficient $\eps^2$.

We verify that $\psi_\eps(x)=\eps\psi^\eps(\frac x\eps)$ so that $\nabla\psi_\eps(x) = \nabla \psi^\eps (\frac x\eps)$, which as we now see is a well defined object in $L^2_{\rm loc}(\Rm^d;\mH)$ uniformly in $\eps$.
\begin{lemma}
\label{lem:psieps}
Let $\psi^\eps$ be the unique solution of \eqref{eq:psiepsy}. We assume that $V$ is such that \eqref{eq:rho} holds.
Then $\eps \psi^\eps$ converges to $0$ in $L^2_{\rm loc}(\Rm^d;\mH)$ as $\eps\to0$. Moreover, $\nabla\psi^\eps(y,\omega)$ converges in $L^2_{\rm loc}(\Rm^d;\mH)$ to a stationary process $\Xi(y,\omega)=\X(\tau_{-y}\omega)$ with $\X\in\mH^d$.

More precisely, we have the estimates for any open domain $D\in\Rm^d$ and $|D|=\int_{D} dx$,
\begin{equation}\label{eq:bounds}
  \|\psi^\eps\|_{L^2(D;\mH)} \leq C\sqrt{|D|},\quad
  \|\psi_\eps\|_{L^2(D;\mH)} \leq C\eps \sqrt{|D|},\quad
  \|\nabla \psi_\eps\|_{L^2(D;\mH)} \leq C\sqrt{|D|}.
\end{equation}
\end{lemma}
\begin{proof}
The equation \eqref{eq:psiepsy} may be equivalently cast as
\begin{equation}
\label{eq:peps}
 (L+\eps^2) \p_\eps + \V=0.
\end{equation}
With $\rD$ the vector valued infinitesimal generators of $T_x$ so that $\rD\cdot \rD=-\LL$ and with $\mH^1$ the Hilbert space of functions $\mf$ in $\mH$ such that $\rD\mf\in (\mH)^d$, we obtain from the Lax-Milgram theory that the above equation admits a unique solution $\p_\eps\in\mH^1$ \cite{PV-RF-81}. Moreover, it is given by
\begin{displaymath}
 \p_\eps = -(\LL+\eps^2)^{-1} \V = \dint_{\Rm^d} \dfrac{-1}{|\xi|^2+\eps^2} U(d\xi)\V.
\end{displaymath}
This shows that
\begin{displaymath}
   \E|\eps\p_\eps|^2 = \eps^2\E |(\LL+\eps^2)^{-1} \V |^2 = \eps^2\E \{\V (\LL+\eps^2)^{-2} \V\} = \dint_{\Rm^d} \dfrac{\eps^2\hat R(\xi)}{|(\xi|^2+\eps^2)^2}  d\xi \leq\rho,
\end{displaymath}
the latter bound coming from separating the contributions $|\xi|<\eps$ and $|\xi|>\eps$. The integrand, which converges to $0$ point-wise, is dominated by $\hat R(\xi)|\xi|^{-2}$. This implies by the dominated Lebesgue convergence theorem that $\E|\eps\p_\eps|^2\to0$ as $\eps\to0$. Similarly,
\begin{displaymath}
 \E|D\p_\eps|^2 = \E |D(\LL+\eps^2)^{-1} \V|^2=\dint_{\Rm^d} \dfrac{|\xi|^2}{|(\xi|^2+\eps^2)^2} \hat R(\xi) d\xi \leq \rho.
\end{displaymath}
By dominated convergence, we thus again observe that $D\p_\eps$ converges to $\X=\int_{\Rm^d}\frac{-i\xi}{|\xi|^2}U(d\xi)\V$ in $\mH$ with $\E|\X|^2=\rho$. It now remains to define $\psi^\eps(y,\omega)=T_y\p_\eps(\omega)$ and $\Xi(y,\omega)=T_y\X(\omega)$ to deduce \eqref{eq:bounds}.
\end{proof}

The above regularity properties of $\psi_\eps$ are not quite sufficient for our convergence proof. We assume more regularity on $V$ and obtain a stronger result on $\psi_\eps$ as follows. 
\begin{corollary}
\label{cor:psieps}
Let us assume that the stationary potential $V$ is such that $V_s:=(-\Delta)^{\frac s2}V$ satisfies the hypotheses of Lemma \ref{lem:psieps}. Then $\nabla\psi_\eps\in L^2(\Omega;H^s_{\rm loc}(\Rm^d))$ for $\psi_\eps$ the solution of \eqref{eq:psieps}.

By Sobolev embedding, then $\nabla \psi^\eps(y)$ is bounded in $L^2(\Omega;L^{2q}_{\rm loc}(\Rm^d))$ with the norm on a bounded domain $D$
\begin{displaymath}
  \Big(\dint_\Omega \big(\dint_D |u(x,\omega)|^{2q} dx\big)^{\frac1q} d\P(\omega)\Big)^{\frac12}
  \leq C \Big(\dint_\Omega \|u(\cdot,\omega)\|_{H^s(D)}^2 d\P(\omega)\Big)^{\frac12}
\end{displaymath}
for $q=\frac{d}{d-2s}$ (and bounded in $L^2(\Omega;L^\infty(\Rm^d))$ when $s>\frac d2$) and converges strongly to its limit $\Xi(y)$ in the $L^2(\Omega;L^{2q}_{\rm loc}(\Rm^d))$ sense. This implies that
\begin{equation}
\label{eq:convsquare}
  \big\| |\nabla\psi_\eps(\frac x\eps)|^2 - |\Xi(\frac x\eps)|^2 \big\| _{L^2(\Omega;L^q(D))}  \xrightarrow{\eps\to0} 0
\end{equation}
for any bounded domain $D$.

By an application of the Birkhoff ergodic theorem, we deduce that
\begin{equation}
\label{eq:ergodicity}
  |\Xi|^2(\frac x\eps) \xrightharpoonup{\eps\to0} \rho \quad \mbox{(weak)} \, L^q_{\rm loc}(\Rm^d),\quad \P-a.s.
\end{equation}

%
\end{corollary}
\begin{proof}
  We observe that
  \begin{displaymath}
    (-\Delta + \eps^2) (-\Delta)^{\frac s2} \psi^\eps + (-\Delta)^{\frac s2}  V=0.
\end{displaymath}
As a consequence, we obtain that $\psi^\eps\in L^2(\Omega;H^s(\Omega))$. The regularity results follow by Sobolev embedding. Then \eqref{eq:convsquare} follows from the result in $L^1$ and the dominated Lebesgue convergence theorem.
\end{proof}

At this stage, \eqref{eq:intpartheta}  may be replaced by
\begin{equation}
\label{eq:intpartheta2}
\E\dint_{\Rm^d} (\nabla u_\eps\cdot\nabla\theta_\eps^* + u_\eps \theta_\eps^* +i\psi_\eps u_\eps\theta_\eps^*+i\nabla\psi_\eps\cdot\nabla (u_\eps\theta_\eps^*)) dx = \E\dint_{\Rm^d} f \theta_\eps^* dx.
\end{equation}
It remains to exhibit the limit of $\nabla\psi_\eps\cdot\nabla u_\eps$, which is non-trivial. In order to do so, we introduce the following perturbed test function
\begin{equation}
\label{eq:perturbedtheta}
\theta_\eps(x,\omega) = \theta(x) e^{i\psi_\eps(x,\omega)},\qquad \theta\in L^2(\Omega;C^\infty_c(\Rm^d)).
\end{equation}
The motivation for the above choice may be explained by formal multi-scale expansions as done in \cite{BLP-77}. Formally assuming that $u_\eps(x)=u_0(x)+\eps u_1(x,y)$, we find that $\eps u_1(x,y)=-u_0(x) \eps i\psi^\eps(y)=-u_0(x)i\psi_\eps(x)$. Moreover,  $1-i\psi_\eps$ is the Taylor expansion of $e^{-i\psi_\eps(x,\omega)}$. Now, the latter quantity is uniformly bounded whereas the former may not be. A similar choice of correctors was considered for a time dependent problem in \cite{HPP-SPDE-13}.

We then obtain that
\begin{displaymath}
  \nabla \theta_\eps = e^{i\psi_\eps}\nabla \theta  + i \theta_\eps\nabla \psi_\eps,\qquad
  \nabla \theta_\eps^* = e^{-i\psi_\eps}\nabla \theta^*  - i \theta_\eps^*\nabla \psi_\eps
\end{displaymath}
We observe that
\begin{displaymath}
\begin{array}{rcl}
&&\nabla u_\eps\cdot\nabla\theta_\eps^* + i\nabla\psi_\eps\cdot\nabla (u_\eps\theta_\eps^*)) \\&=&
    e^{-i\psi_\eps}\nabla u_\eps\cdot\nabla\theta^* -i \theta_\eps^* \nabla u_\eps \cdot\nabla\psi_\eps
    +i \theta_\eps^*\nabla\psi_\eps\cdot\nabla u_\eps
    + i e^{-i\psi_\eps} u_\eps\nabla\psi_\eps\cdot\nabla \theta^*
    +\theta_\eps^* u_\eps |\nabla\psi_\eps|^2 \\&=&
    e^{-i\psi_\eps}\nabla u_\eps\cdot\nabla\theta^* + i e^{-i\psi_\eps} u_\eps\nabla\psi_\eps\cdot\nabla \theta^*
    +\theta_\eps^* u_\eps |\nabla\psi_\eps|^2.
\end{array}
\end{displaymath}
We may now recast \eqref{eq:intpartheta2} as
\begin{equation}
\label{eq:intpartheta3}
\E\dint_{\Rm^d} \big(e^{-i\psi_\eps}\nabla u_\eps\cdot\nabla\theta^* + u_\eps \theta_\eps^* (1+i\psi_\eps)+ i e^{-i\psi_\eps} u_\eps\nabla\psi_\eps\cdot\nabla \theta^*
    +\theta_\eps^* u_\eps |\nabla\psi_\eps|^2 - f \theta_\eps^*\big) dx = 0.
\end{equation}
It remains to pass to the limit in each of the terms above. Since $|e^{-i\psi_\eps}-1|\leq|C\psi_\eps|$, we deduce from lemma \ref{lem:psieps} that $\theta\psi_\eps$ converges to $0$ in $L^2(\Rm^d;\mH)$ and hence that $\theta_\eps^*=e^{-i\psi_\eps}\theta^*$ and $e^{-i\psi_\eps}\nabla\theta^*$ converge to $\theta^*$ and $\nabla\theta^*$, respectively, in the same sense. Similarly, $\psi_\eps\theta_\eps^*$ converges to $0$ in the same sense. This shows that
\begin{displaymath}
  \E\dint_{\Rm^d} \big(e^{-i\psi_\eps}\nabla u_\eps\cdot\nabla\theta^* + u_\eps \theta_\eps^* (1+i\psi_\eps) - f\theta^*_\eps \big) dx \xrightarrow{\eps\to0}
 \E \dint_{\Rm^d}\big( \nabla u\cdot\nabla\theta^* + u\theta -f\theta\big) dx.
\end{displaymath}

Let us consider the term $T_1=i e^{-i\psi_\eps} u_\eps\nabla\psi_\eps\cdot\nabla \theta^*$. On the support $D$ of $\theta$, $e^{-i\psi_\eps}\nabla\psi_\eps$ is bounded in the $L^2(D;\mH)$ sense. Since $(u_\eps-u)\nabla\theta^*$ converges to $0$ in that sense, the limit of the integral of $T_1$ is the same as that of $T_2=i e^{-i\psi_\eps} u\nabla\psi_\eps\cdot\nabla \theta^*$. For the same reason, we may now replace $e^{-i\psi_\eps}$ by its limit $1$ so the limit of the integral of $T_1$ is the same as that of $T_3=iu\nabla\psi_\eps\cdot\nabla \theta^*$, and by integrations by parts the same as that of
$T_4=-i\psi_\eps\nabla\cdot(u\nabla\theta^*)$. Since $\psi_\eps$ goes to $0$ in $L^2_{\rm loc}$ and $\theta$ is smooth and compactly supported, we obtain that
\begin{displaymath}
    \E\dint_{\Rm^d} \big( i e^{-i\psi_\eps} u_\eps\nabla\psi_\eps\cdot\nabla \theta^* \big) dx \xrightarrow{\eps\to0}0.
\end{displaymath}

Finally, we consider the convergence of the integral of $\theta_\eps^* u_\eps |\nabla\psi_\eps|^2$. We have that $\theta^*_\eps u_\eps$ converges strongly to $\theta u$ in $L^p$ for $1<p<\frac {2d}{d-2}$.  We thus need some regularity on $|\nabla\psi_\eps|^2(x)=|\nabla\psi^\eps|^2(\frac x\eps)$.  From corollary \ref{cor:psieps}, we deduce that $\nabla\psi_\eps(\frac x\eps)$ is bounded in $L^{2q}_{\rm loc}$ for $2\leq q\leq\frac {d}{d-2s}$ (or $L^\infty$ when $s>\frac d2$) and hence that $|\nabla \psi_\eps|^2(\frac x\eps)$ is bounded in $L^q_{\rm loc}$. Choosing $q=p'>\frac {2d}{d+2}$, which holds when $s>\frac{d-2}4$, we obtain from \eqref{eq:convsquare} that  the integral of $\theta_\eps^* u_\eps |\nabla\psi_\eps|^2$ has the same limit as the integral of $\theta u |\Xi|^2(\frac x\eps)$.

Since $\theta u\in L^2(\Omega;L^p(\Rm^d))$, we obtain from the Birkhoff ergodic theorem in \eqref{eq:ergodicity} that
\begin{equation}
\label{eq:ergodicity2}
  \E\dint_{\Rm^d} u\theta|\Xi|^2(\frac x\eps) dx \xrightarrow{\eps\to0} \E \dint_{\Rm^d}  u\theta \rho dx
\end{equation}

This shows that for all $\theta\in L^2(\Omega;C^\infty_c(\Rm^d))$, we have that
\begin{displaymath}
   \E\dint_{\Rm^d} \big( \nabla u\cdot\nabla \theta + (1+\rho) u\theta - f\theta\big) dx =0.
\end{displaymath}
This is the weak formulation in $H^1(\Rm^d;\mH)$ (valid for all $\theta\in L^2(\Omega;C^\infty_c(\Rm^d))$ and by density for all $\theta\in H^1(\Rm^d;\mH)$) of the (unique and deterministic) solution to the equation \eqref{eq:limit}.
This proves theorem \ref{thm:conv}.

%
%

\section{Decorrelation properties and rate of convergence}
\label{sec:rate}

We now prove theorem \ref{thm:rate}. Our main assumption on the coefficients is a control of the fourth-order moments of the potential $V(x)$ as well as some regularity on the unique solution $u_0$ of the limiting equation. More precisely, we assume that $f\in L^2(\Rm^d)$ and denote by $u_0$ the solution in $H^2(\Rm^d)$ of \eqref{eq:limit}.

Let $G$ be the Green's function defined as the fundamental solution of $(-\Delta+1)G(x)=\delta(x)$. It is given by the explicit expression $G(x)=c_ne^{-|x|}|x|^{2-n}$ for a normalizing constant $c_n>0$. Then we find that for $\nu>0$ and $C>0$ that
\begin{equation}
\label{eq:bds}  G(x)|x|+|\nabla G|(x) \leq C \frac{\exp(-\nu|x|)}{|x|^{d-1}}.
\end{equation}

Define $\chi_{\varepsilon} = G*(\frac{-i}{\varepsilon^2}V(\frac{\cdot}{\varepsilon})) = \frac i\eps \psi_\eps$ and
$u_1(x) = -\chi_{\varepsilon}(x)u_0(x)$. Some algebra shows that
\begin{equation}\label{eq:veps}
    (\Delta-1+iV_\eps)(u_0+\varepsilon u_1-u_{\varepsilon}) = (\rho-iV(\frac{x}{\varepsilon})\chi_{\varepsilon}(x))u_0-\varepsilon(\chi_{\varepsilon}\Delta u_0 + 2\nabla\chi_{\varepsilon}\cdot\nabla u_0).
\end{equation}
In other words, $u_0+\eps u_1 = u_0(1-\eps\chi_\eps)$ is the leading expansion of $u_\eps$. In the preceding section, we proved that $\eps\chi_\eps$ converged to $0$ in the $L^2(D;\mH)$ sense for $D$ a bounded domain. We also observe that $\nabla (u_0\eps u_1)$ is well approximated by  $\nabla u_0 - u_0\Xi(\frac \cdot\eps)$.

When the potential $V$ decorrelates sufficiently rapidly, then we can obtain optimal rates of convergence of $u_\eps$ to $u_0$ in $L^2(\Rm^d)$ and error estimates between $u_\eps$ and $u_0+\eps u_1$ in $H^1(\Rm^d)$.

Let us assume that the correlation function $R(x)$ is integrable.
 Then the size of $\varepsilon u_1$ may be estimated as
\begin{equation}
    \label{u1}
    \begin{aligned}
        \mathbb{E}\int |\varepsilon u_1(x)|^2 dx &= \frac{1}{\varepsilon^2}\int G(x-y_1)G(x-y_2) R(\frac{y_1-y_2}{\varepsilon}) |u_0(x)|^2 dy_1dy_2dx\\
        &=\frac{1}{\varepsilon^2}\int G(y_1)G(y_2)R(\frac{y_1-y_2}{\varepsilon})|u_0(x)|^2dy_1dy_2dx\\
        &\leq \frac{1}{\varepsilon^2}\int\frac{\exp(-\nu|y_1|)}{|y_1|^{d-2}}\frac{\exp(-\nu|y_2|)}{|y_2|^{d-2}}R(\frac{y_1-y_2}{\varepsilon})dy_1dy_2\int |u_0(x)|^2 dx\\
        &=\frac{1}{\varepsilon^2}\int\frac{\exp(-\nu|y_1|)}{|y_1|^{d-2}}\frac{\exp(-\nu|y_1-y_2|)}{|y_1-y_2|^{d-2}}R(\frac{y_2}{\varepsilon})dy_1dy_2\int |u_0(x)|^2 dx\\
        &\leq\left\{
        \begin{array}{lll}
            \frac{1}{\varepsilon^2}\|u_0\|_2^2\int C e^{-\nu|y_2|} R(\frac{y_2}{\varepsilon})dy_2&=O(\eps) & d=3\\
            \frac{1}{\varepsilon^2}\|u_0\|_2^2\int C e^{-\nu|y_2|}(\log|y_2|+1)R(\frac{y_2}{\varepsilon})dy_2&=O(\eps^2|\ln\eps|) & d=4\\
            \frac{1}{\varepsilon^2}\|u_0\|_2^2\int C e^{-\nu|y_2|}(|y_2|^{-(d-4)}+1)R(\frac{y_2}{\varepsilon})dy_2&=O(\eps^2) & d>4.
        \end{array}
        \right.
\end{aligned}
\end{equation}
The latter estimates easily follow from the integrability of the correlation function in dimension $d=3$ and $d=4$. For $d>4$, we decompose the integral into two parts as
\begin{equation}
\begin{aligned}
\frac{1}{\varepsilon^2}\int e^{-\nu|y_2|} (|y_2|^{-(d-4)}+1)R(\frac{y_2}{\varepsilon})dy_2 &= \frac{1}{\varepsilon^2}\int_{|y_2|\geq 1} e^{-\nu|y_2|} (|y_2|^{-(d-4)}+1)R(\frac{y_2}{\varepsilon})dy_2\\
&+ \frac{1}{\varepsilon^2}\int_{|y_2|<1} e^{-\nu|y_2|} (|y_2|^{-(d-4)}+1)R(\frac{y_2}{\varepsilon})dy_2.
\end{aligned}
\end{equation}
We recast this as $ (i)+(ii)$ and
$(i)$ and $(ii)$ are estimated respectively as
\begin{equation}
(i)\leq \frac{2}{\varepsilon^2}\int_{|y_2|\geq 1} \exp(-\nu |y_2|) R(\frac{y_2}{\varepsilon}) dy_2\leq 2\varepsilon^{d-2}\exp(-\nu) \|R\|_1,
\end{equation}
\begin{equation}
\begin{aligned}
(ii)&\leq \frac{1}{\varepsilon^2}\int_{|y_2|\leq 1}(|y_2|^{-(d-4)}+1)R(\frac{y_2}{\varepsilon})dy_2\leq \frac{2}{\varepsilon^2}\int_{|y_2|\leq 1}|y_2|^{-(d-4)}R(\frac{y_2}{\varepsilon})dy_2\\
&= 2\varepsilon^2\int_{|y_2|\leq \frac{1}{\varepsilon}}|y_2|^{-(d-4)}R(y_2)dy_2 \leq 2\varepsilon^2\int |y_2|^{-(d-4)}R(y_2)dy_2 \leq 2\varepsilon^2 \|R\|_1.
\end{aligned}
\end{equation}
By replacing the Green's function with its gradient in \eqref{u1} we find that
\begin{equation}\label{nabla_u1}
\mathbb{E}\int |\varepsilon \nabla u_1(x)|^2 dx \sim O(1).
\end{equation}
This shows that $\eps u_1$ is negligible in the $L^2$ sense but not in the $H^1$ sense. We now estimate the error $v_{\varepsilon}:=u_0+\varepsilon u_1-u_{\varepsilon}$ using \eqref{eq:veps}. Multiplying \eqref{eq:veps} by $-v_\eps^*$ and integrating by parts, we know from the analysis in the preceding section that
\begin{equation}
\label{eq:estim}
\|v_\eps  \|^2_{H^1(\mathbb{R}^d;\mH)}  \leq \Big| \mathbb{E}\int(\rho-iV(\frac{x}{\varepsilon})\chi_{\varepsilon}(x))u_0v^*_{\varepsilon}dx \Big|+\Big| \varepsilon\mathbb{E}\int(\chi_{\varepsilon}\Delta u_0 + 2\nabla\chi_{\varepsilon}\cdot\nabla u_0)v^*_{\varepsilon}dx\Big|.
\end{equation}
Let us consider the second-term on the above right-hand side. The term $\chi_\varepsilon \Delta u_0$ can be estimated in the same way as $u_1$ and by using the Cauchy-Schwarz inequality, we have
\begin{equation}\label{right1}
\left|\varepsilon\mathbb{E}\int \chi_{\varepsilon}\Delta u_0 v^*_\varepsilon dx\right|\leq C\|v_{\varepsilon}\|_{L^2(\mathbb{R}^d;\mH)}\times\left\{
\begin{array}{ll}
\sqrt{\varepsilon},& d=3\\
\varepsilon\sqrt{\log\varepsilon},& d=4\\
\varepsilon, & d>4.
\end{array}
\right.
\end{equation}
The integral $\varepsilon\mathbb{E}\int\nabla\chi_{\varepsilon}\cdot\nabla u_0 v^*_\varepsilon dx$ is estimated using integrations by parts as
\begin{equation}\label{right2}
\begin{aligned}
\left|\varepsilon\mathbb{E}\int\nabla\chi_{\varepsilon}\cdot\nabla u_0 v^*_\varepsilon dx\right|
&=\left|\varepsilon \mathbb{E}\int(\nabla\cdot(\chi_{\varepsilon}\nabla u_0)-\chi_{\varepsilon}\Delta u_0)v^*_\varepsilon  dx\right| \\
&\leq \left|\varepsilon \mathbb{E}\int \nabla {v}_{\varepsilon}^*\cdot\nabla u_0 \chi_{\varepsilon}dx \right|
+\left|\varepsilon\mathbb{E}\int \chi_{\varepsilon}\Delta u_0 v^*_\varepsilon   dx\right|\\
&\leq C\|\nabla v_{\varepsilon}\|_{L_2(\mathbb{R}^d;\mH)}\times\left\{
\begin{array}{ll}
\sqrt{\varepsilon},& d=3\\
\varepsilon\sqrt{|\ln\varepsilon|},& d=4\\
\varepsilon, & d>4.
\end{array}
\right.
\end{aligned}
\end{equation}
The first term on the right-hand side in \eqref{eq:estim} is bounded by
\begin{displaymath}
  \|v_\eps  \|_{H^1(\mathbb{R}^d;\mH)}\big\| (\rho-iV(\frac{x}{\varepsilon})\chi_{\varepsilon}(x))u_0\|_{H^{-1}(\mathbb{R}^d;\mH)}.
\end{displaymath}
Recalling that $G$ is the fundamental solution of $-\Delta+1$, we obtain that
\begin{equation}
\label{eq:rhsestim}
\big\| (\rho-iV(\frac{x}{\varepsilon})\chi_{\varepsilon}(x))u_0\|_{H^{-1}(\mathbb{R}^d;\mH)}=\big\| G*\big((\rho-iV(\frac{x}{\varepsilon})\chi_{\varepsilon}(x))u_0\big) \big\| _{H^{1}(\mathbb{R}^d;\mH)},
\end{equation}
since $-\Delta+1$ is an isomorphism from $H^1(\Rm^d)$ to $H^{-1}(\Rm^d)$.

Define $f_{\varepsilon}(x) = G*((\rho-iV(\frac{x}{\varepsilon})\chi_{\varepsilon}(x))u_0)$.  We show in the next section that $\|f_\eps\|_{H^1(\Rm^d;\mH)}$ is bounded by a constant times $\sqrt\eps$ in $d=3$, $\eps|\ln\eps|^{\frac12}$ in $d=4$ and $\eps$ in $d>4$. Note that $\rho=\lim_{\eps\to0}\E\{iV(\frac x\eps)\chi_\eps(x)\}$ so that $f_\eps$ is asymptotically mean-zero.
%
%

Collecting the previous bounds, we obtain that
\begin{equation}
\|u_\eps - u_0-\varepsilon u_1\|_{H^1(\mathbb{R}^d;\mH)} + \|u_\eps -u_0\|_{L^2(\mathbb{R}^d;\mH)} \leq C \left\{
\begin{array}{ll}
\sqrt{\varepsilon} & d=3\\
\varepsilon\sqrt{|\ln\varepsilon|} & d=4\\
\varepsilon & d>4.
\end{array}
\right.
\end{equation}
This concludes the proof of theorem \ref{thm:rate}.

\section{Estimation of fourth order moments}
\label{sec:fourth}
In this section we discuss the estimation of $\mathbb{E}\int|\nabla
f_{\varepsilon}|^2dx$ and $\mathbb{E}\int|f_{\varepsilon}|^2dx$
when the potential $V$ satisfies \eqref{eq:estim4th}. Following \cite{HPP-SPDE-13}, we first recall that the latter estimate holds for a large class of sufficiently mixing coefficients.
\begin{definition}
For any $r>0$, $\gamma(r)$ is the smallest value such that the bound
\begin{equation}
\mathbb{E}(\phi_1(V)\phi_2(V))\leq
\gamma(r)\sqrt{\mathbb{E}\phi_1^2(V)\mathbb{E}\phi_2^2(V)},
\end{equation}
holds for any two compact sets $K_1$, $K_2$ such that
\begin{equation}
d(K_1,K_2) = \inf_{x_1\in K_1,x_2\in K_2}(|x_1-x_2|)\geq r,
\end{equation}
for any two random variables $\phi_i(V)$ such that $\phi_i(V)$ is
$\mathcal{F}_{K_i}$-measurable and $\mathbb{E}\phi_i(V)=0$.
\end{definition}

It is shown in \cite{HPP-SPDE-13} that \eqref{eq:estim4th} holds for 
a function $\eta:\ \mathbb{R}_{+}\rightarrow \mathbb{R}_{+}$ defined by
\begin{equation}
\eta(r)=\sqrt{K\gamma(r/3)},\qquad \text{with}\
K=4(\|V(x)\|_2\|V^3(x)\|_2+\|V^2(x)\|_2^2).
\end{equation}
Note that when $V(\cdot)$ is a Gaussian random field, inequality
\eqref{eq:estim4th} becomes an equality with $\eta$ replaced by
$R$. We assume that $\eta\in L^1(\mathbb{R}^d)$, and hence that $\sqrt\gamma\in L^1(\Rm^d)$ for the following
estimation to hold.

We have the following decomposition for $\|\nabla f_{\varepsilon}\|_{L^2(\Rm^d;\mH)}^2$
\begin{equation}
    \label{grad_4thmoment}
\begin{aligned}
        &\mathbb{E}\int|\nabla f_{\varepsilon}|^2dx\\
        =&\mathbb{E}\left|\rho\int \nabla G(x,y)u_0(y)dy-\frac{1}{\varepsilon^2}\int\int \nabla G(x,y)V(\frac{y}{\varepsilon})G(y,z)V(\frac{z}{\varepsilon})u_0(y)dydz\right|^2dx\\
        =&\frac{1}{\varepsilon^4}\mathbb{E}\int \nabla G(x,y_1)G(y_1,z_1) \nabla G(x,y_2)G(y_2,z_2)V(\frac{y_1}{\varepsilon})V(\frac{z_1}{\varepsilon})V(\frac{y_2}{\varepsilon})V(\frac{z_2}{\varepsilon})u_0(y_1)u_0(y_2)d{\cal Y}\\
        -&\frac{1}{\varepsilon^4}\int \nabla G(x,y_1)G(y_1,z_1) \nabla G(x,y_2)G(y_2,z_2)R(\frac{y_1-z_1}{\varepsilon})R(\frac{y_2-z_2}{\varepsilon})u_0(y_1)u_0(y_2) d{\cal Y}\\
    \leq &\frac{1}{\varepsilon^4}\int \nabla G(x,y_1)G(y_1,z_1) \nabla G(x,y_2)G(y_2,z_2)\eta(\frac{y_1-y_2}{\varepsilon})\eta(\frac{z_1-z_2}{\varepsilon})u_0(y_1)u_0(y_2)d{\cal Y}\\
    +&\frac{1}{\varepsilon^4}\int \nabla G(x,y_1)G(y_1,z_1) \nabla G(x,y_2)G(y_2,z_2)\eta(\frac{y_1-z_2}{\varepsilon})\eta(\frac{y_2-z_1}{\varepsilon})u_0(y_1)u_0(y_2)d{\cal Y}\\
    :=&(I)+(II)\\
\end{aligned}
\end{equation}
with $d{\cal Y}=dy_1 dy_2 dz_1 dz_2dx$.\\
{\itshape Estimation of (I).} Changing variables $y_i$ and $z_i$
to $x-y_i$ and $x-y_i-z_i$ for $i=1,\ 2$ gives
\begin{equation}
    \begin{aligned}
        |(I)|\leq &C\frac{1}{\varepsilon^4}\int \nabla G(y_1) G(z_1) \nabla G(y_2) G(z_2)|\eta(\frac{y_1-y_2}{\varepsilon})||\eta(\frac{y_1-y_2}{\varepsilon}-\frac{z_1-z_2}{\varepsilon})|\\
        &\qquad \qquad \times |u_0(x-y_1)||u_0(x-y_2)|dy_1dy_2dz_1dz_2dx.
\end{aligned}
\end{equation}
Using $u_0$ to integrate in $x$, we then have
\begin{equation}
    \begin{aligned}
    |(I)|\leq &C\frac{1}{\varepsilon^4}\int \nabla G(y_1)G(z_1)\nabla G(y_2) G(z_2)|\eta(\frac{y_1-y_2}{\varepsilon})||\eta(\frac{y_1-y_2}{\varepsilon}-\frac{z_1-z_2}{\varepsilon})|dy_1dy_2dz_1dz_2.
\end{aligned}
\end{equation}
Changing variables $y_2$ and $z_2$ to $y_1-y_2$ and $z_1-z_2$, and using \eqref{eq:bds} yields
\begin{equation}
\begin{aligned}
    |(I)|\leq &C\frac{1}{\varepsilon^4}\int \frac{\exp(-\nu|y_1|)}{|y_1|^{d-1}}\frac{\exp(-\nu|z_1|)}{|z_1|^{d-2}}\frac{\exp(-\nu|y_1-y_2|)}{|y_1-y_2|^{d-1}}\frac{\exp(-\nu|z_1-z_2|)}{|z_1-z_2|^{d-2}}\\
    &\hspace{2cm}\times |\eta(\frac{y_2}{\varepsilon})||\eta(\frac{y_2}{\varepsilon}-\frac{z_2}{\varepsilon})|dy_1dy_2dz_1dz_2.
\end{aligned}
\end{equation}
Now we may apply Lemma \ref{Wenjia_Lemma} to integrate in $y_1$ and
$z_1$:
\begin{eqnarray}
\int\frac{\exp(-\nu|y_1|)}{|y_1|^{d-1}}\frac{\exp(-\nu|y_1-y_2|)}{|y_1-y_2|^{d-1}}dy_1 \leq C\exp(-\nu |y_2|)(1+|y_2|^{-(d-2)}),\\
\int\frac{\exp(-\nu|z_1|)}{|z_1|^{d-2}}\frac{\exp(-\nu|z_1-z_2|)}{|z_1-z_2|^{d-2}}dz_1
\leq \left\{
\begin{array}{ll}
C\exp(-\nu z_2), & d=3\\
C\exp(-\nu z_2)(1+\log|z_2|), & d=4\\
C\exp(-\nu z_2)(1+|z_2|^{-(d-4)}), & d>4.
\end{array}
\right.
\end{eqnarray}
This estimate of (I) can be recast as
\begin{equation}
\begin{aligned}
|(I)|\leq &C\frac{1}{\varepsilon^4}\int\int dy_2dz_2\exp(-\nu|y_2|)(1+|y_2|^{-(d-2)}) \exp(-\nu|z_2|) |\eta(\frac{y_2}{\varepsilon})||\eta(\frac{y_2-z_2}{\varepsilon})|\\
&\times\left\{
\begin{array}{ll}
1, & d=3 \\
\log(|z_2|), & d=4 \\
(1+|z_2|^{-(d-4)}), & d>4.
\end{array}
\right.
\end{aligned}
\end{equation}
It remains to integrate in $y_2$ and $z_2$ to obtain
\begin{equation}\label{estI}
|(I)| \sim\left\{
\begin{array}{ll}
    O(\varepsilon), &d=3\\
    O(\varepsilon^2|\log\varepsilon|), &d=4\\
    O(\varepsilon^2), &d>4.
\end{array}
\right.
\end{equation}
{\itshape Estimation of (II).} After changing variables $y_i$ and
$z_i$ to $x-y_i$ and $x-y_i-z_i$ for $i=1,\ 2$, and integrating in
$x$ using $u_0$, we have
\begin{equation}
    \begin{aligned}
    |(II)|\leq &C\frac{1}{\varepsilon^4}\int \nabla G(y_1)G(z_1)\nabla G(y_2) G(z_2)|\eta(\frac{-y_1+y_2+z_2}{\varepsilon})||\eta(\frac{-y_2+y_1+z_1}{\varepsilon})|d{\cal Y}
\end{aligned}
\end{equation}
with $d{\cal Y}=dy_1dy_2dz_1dz_2$. Changing variable $y_2$ to $y_1-y_2$ and using 
\eqref{eq:bds} gives
\begin{equation}
\begin{aligned}
    |(II)|\leq &C\frac{1}{\varepsilon^4}\int \frac{\exp(-\nu|y_1|)}{|y_1|^{d-1}}\frac{\exp(-\nu|z_1|)}{|z_1|^{d-2}}\frac{\exp(-\nu|y_1-y_2|)}{|y_1-y_2|^{d-1}}\frac{\exp(-\nu|z_2|)}{|z_2|^{d-2}}\\
    &\hspace{2cm} |\eta(\frac{z_2-y_2}{\varepsilon})||\eta(\frac{z_1+y_2}{\varepsilon})|dy_1dy_2dz_1dz_2.
\end{aligned}
\end{equation}
We now integrate in $y_1$ and $z_1$:
\begin{equation}
\begin{aligned}
 \int\frac{\exp(-\nu|y_1|)}{|y_1|^{d-1}}\frac{\exp(-\nu|y_1-y_2|)}{|y_1-y_2|^{d-1}}dy_1 &\leq C\exp(-\nu |y_2|)(1+|y_2|^{-(d-2)}),\\
 \int\frac{\exp(-\nu|z_1|)}{|z_1|^{d-2}}|\eta(\frac{z_1+y_2}{\varepsilon})|dz_1 &\leq C\varepsilon^2.
\end{aligned}
\end{equation}
The estimate is then recast as
\begin{equation}
|(II)|\leq C\frac{1}{\varepsilon^2}\int\int \exp(-\nu
|y_2|)(1+|y_2|^{-(d-2)})\frac{\exp(-\nu|z_2|)}{|z_2|^{d-2}}||\eta(\frac{z_2-y_2}{\varepsilon})|dy_2dz_2.
\end{equation}
Changing variable $z_2$ to $y_2-z_2$, and integrating in $y_2$ using
Lemma \ref{Wenjia_Lemma} yields
\begin{equation}
\begin{aligned}
|(II)|\leq &C\frac{1}{\varepsilon^2}\int dz_2\exp(-\nu|z_2|) |\eta(\frac{z_2}{\varepsilon})|
&\times\left\{
\begin{array}{ll}
1, & d=3 \\
\log(|z_2|), & d=4 \\
(1+|z_2|^{-(d-4)}), & d>4.
\end{array}
\right.
\end{aligned}
\end{equation}
It remains to integrate in $z_2$ to obtain
\begin{equation}\label{estII}
|(II)| \sim\left\{
\begin{array}{ll}
    O(\varepsilon), &d=3\\
    O(\varepsilon^2|\log\varepsilon|), &d=4\\
    O(\varepsilon^2), &d>4.
\end{array}
\right.
\end{equation}
Collecting \eqref{estI} and \eqref{estII}, we find that
\begin{equation}
 \mathbb{E}\int |\nabla f_{\varepsilon}|^2 dx \sim \left\{
\begin{array}{ll}
O(\varepsilon) & d=3\\
O(\varepsilon^2|\log\varepsilon|) & d=4\\
O(\varepsilon^2) & d>4.
\end{array}
\right.
\end{equation}
The estimate of $\mathbb{E}\int | f_{\varepsilon}|^2 dx$ can be
obtained by replacing $\nabla G$ by $G$ in \eqref{grad_4thmoment}
and estimating every term in the same way. The result is
\begin{equation}
 \label{4thmoment}
 \mathbb{E}\int |f_{\varepsilon}|^2 dx \sim \left\{
\begin{array}{ll}O(\varepsilon^2) & d=3\\

O(\varepsilon^4|\log\varepsilon|^2) & d=4\\
O(\varepsilon^4) & d>4.
\end{array}
\right.
\end{equation}
This concludes the proof of theorem \ref{thm:rate}.

\section*{Acknowledgments}
The authors would like to thank Yu Gu or multiple discussions on the homogenization of equations with random potentials. This work was partially funded by AFOSR Grant NSSEFF- FA9550-10-1-0194 and NSF Grant DMS-1108608.
\appendix
\section{Appendix}
The following lemma is proved in \cite{BJ-CMS-11}.
\begin{lemma}\label{Wenjia_Lemma}
Let us fix two distinct points $x,y\in \mathbb{R}^d$. Let $\alpha$, $\beta$ be positive numbers in $(0,d)$, and $\lambda$ another positive number. We have the following convolution results.
\begin{equation}
\int_{\mathbb{R}^d} \frac{e^{-\lambda|z-x|}}{|z-x|^{\alpha}}\frac{e^{-\lambda|z-y|}}{|z-y|^{\beta}}dz \leq \left\{
\begin{array}{ll}
C\exp(-\lambda|x-y|)(|x-y|^{d-(\alpha+\beta)}+1), & \text{if}\ \alpha+\beta>d;\\
C\exp(-\lambda|x-y|)(|\log|x-y||+1), & \text{if}\ \alpha+\beta=d;\\
C\exp(-\lambda|x-y|), & \text{if}\ \alpha+\beta<d.
\end{array}
\right.
\end{equation}
The above constants depend only on the diam(X), $\alpha$, $\beta$, $\lambda$, and dimension $d$ but not on $|x-y|$.
\end{lemma}

%
%
%
%
%
%
%
%
%
%
%
%


\end{document}